\documentclass[11pt]{amsart}
\usepackage{amssymb,latexsym,graphicx, amscd, enumitem}
\newtheorem{theorem}{Theorem}[section]
\newtheorem{prop}[theorem]{Proposition}
\newtheorem{lemma}[theorem]{Lemma}

\newtheorem{definition}[theorem]{Definition}

\newtheorem{example}[theorem]{Example}

\DeclareMathOperator{\ep}{\epsilon}

\DeclareMathOperator{\C}{\mathbb C}

\begin{document}

\title{Images and Singularity Subsets of Pseudoholomorphic Maps}

\author{Weiyi Zhang}
\address{Mathematics Institute\\  University of Warwick\\ Coventry, CV4 7AL, England}
\email{Weiyi.Zhang@warwick.ac.uk}

\begin{abstract} We study the image and the singularity subset of a general pseudoholomorphic map. We show that the image of a proper pseudoholomorphic map is a pseudoholomorphic subvariety when the dimension of either the domain or target is four. We also prove that the singularity subset of a pseudoholomorphic map is pseudoholomorphic  when the domain has dimension four. 
\end{abstract}
\maketitle

\section{Introduction}
A pseudoholomorphic curve is a smooth map from a Riemann surface into an almost complex manifold satisfying the Cauchy-Riemann equation. 
In an almost complex manifold, the image of a pseudoholomorohic curve or an almost complex submanifold is not the zero locus of a pseudoholomorphic map even locally in general. Hence, unlike in algebraic geometry where lots of subvarieties are generated from divisors and applying intersection theory, pseudoholomorphic curves are mostly studied from the above ``mapping into" viewpoint. 

In \cite{ZintJ}, we have developed the ``mapping out" approach for pseudoholomophic curves, or more generally pseudoholomorphic subvarieties. Namely, we study the intersection theory of almost complex submanifolds. In this paper, we would show that some part of the arguments in \cite{ZintJ} also help us to understand the original ``mapping into" definition of pseudoholomorphic subvarieties.

Recall that a $J$-holomorphic subvariety of an almost complex manifold $(M, J)$ is a finite set of pairs $\Theta=\{(V_i, m_i), 1\le i\le m\}$, where each $V_i$ is an irreducible $J$-holomorphic subvariety and each $m_i$ is a positive integer. Here an irreducible $J$-holomorphic subvariety of $(M, J)$ is the image of a somewhere immersed 
pseudoholomorphic map $\phi: X\rightarrow M$ from a compact connected smooth almost complex manifold $X$. If each irreducible component $V_i$ has the same complex dimension $n$, we call $\Theta$ a $J$-holomorphic $n$-subvariety, or simply  $n$-subvariety.

An outstanding question about this definition is whether the image of a general pseudoholomorphic map is still a $J$-holomorphic subvariety? In the complex setting, this is a theorem of Remmert \cite{Rem} which says that the image of a proper holomorphic map between two complex manifolds is a complex analytic subvariety of the target. In fact, proper mapping theorem is true when this is a pseudoholomorphic map from a closed almost complex manifold to a complex manifold (see \cite{CZ2}).

In this paper, we study this question for a general pseudoholomorphic map between almost complex manifolds. Our first result is 

\begin{theorem}\label{to4intro}
Let $(M, J)$ be a connected almost complex $4$-manifold, $(X, J_X)$ a connected closed almost complex manifold, and $f: X\rightarrow M$ a proper pseudo\-holomorphic map. If  $f$ is not a constant map and the image $f(X)\ne M$, then $f(X)$ is an irreducible $J$-holomorphic $1$-subvariety.    
\end{theorem}

A similar argument would also lead to the following

\begin{theorem}\label{from4intro}
Let $(M, J)$ be a connected almost complex manifold, $(X, J_X)$ a connected closed almost complex $4$-manifold, and $f: X\rightarrow M$ a proper pseudoholomorphic map. Then $f(X)$ is an irreducible $J$-holomorphic subvariety.    
\end{theorem}

These two results will be proved in Section $3$ as Theorems \ref{to4} and \ref{from4}. We apply a similar strategy to prove both results: we first locally foliate $X$ by pseudoholomoprhic disks, then use Lemma \ref{tandef} to show that the leaves are unique continuation of each other.  

We then study the singularity subset of a somewhere immersed pseudoholomophic map. Recall the singularity subset of a map $f: X\rightarrow M$ is the set where the differential $df_p$ is not of full rank. In particular, we get the following, which is Theorem \ref{sinJhol} in the paper. 
\begin{theorem}\label{sinJholintro}
Let $(M, J)$ be a connected almost complex manifold, $(X, J_X)$ a connected closed almost complex $4$-manifold, and $f: X\rightarrow M$ a pseudoholomorphic map. If $f$ is somewhere immersed, then $f$ is immersed in an open dense subset of $X$ and its singularity subset is a union of finitely many $J_X$-holomorphic $1$- and $0$-subvarieties. In particular, it implies that other than a union of finitely many $1$- and $0$-subvarieties subset, a $2$-subvariety is a smooth almost complex $4$-manifold. 
\end{theorem}

Similar results also hold when we remove the assumption that $f$ is somewhere immersed (Theorem \ref{sinJh}). 

This result is under the guiding philosophy of \cite{ZintJ}: A statement for smooth maps between smooth manifolds in terms of  Ren\'e Thom's transversality should also have its counterpart in the pseudoholomorphic setting without requiring the transversality or genericity, but using the notion of pseudoholomorphic subvarieties, 
in particular when such a statement is available in the complex analytic setting.
Its corresponding results in the smooth setting is Sard's theorem and the description of the singularity subset for a generic map by Thom and Boardman \cite{Boa, Th}. 

In \cite{ZintJ}, we have shown this theorem when $\dim X=\dim M=4$. In fact, the singularity subset must be a $J_X$-holomorphic $1$-subvariety in this case.  The basic strategy is to express the singularity subset as the intersection of two almost complex submanifolds in a larger almost complex manifold. In the case of $\dim X=\dim M=4$, these two almost complex submanifolds have dimension $4$ and codimension $2$ respectively, in the situation that we can apply Theorem 1.2 of \cite{ZintJ}.

In the more general setting of Theorem \ref{sinJholintro}, we can still reduce the problem to the intersection of almost complex submanifolds. However, it is now in the situation of excess intersection, which is not covered in \cite{ZintJ}. Here, we apply the projection to coordinates strategy used in \cite{CZ2} and express the singularity subset as the intersection of finitely many subsets in $X$ which are obtained as intersection with almost complex divisors in different ambient almost complex manifolds. 

Although this strategy is very useful in showing excess intersection also gives pseudoholomorphic subvarieties in many applications, it does not apply to the most general situation. The main reason is that we do not have standard neighborhood theorem as in symplectic setting. In general, a tubular neighborhood of an almost complex submanifold in an ambient almost complex manifold does not have the structure of a pseudoholomorphic bundle, even in the complex setting. 

 We would like to thank Mario Micallef for indispensable discussions. In particular, the argument of Lemma \ref{tandef} is suggested by him.

\section{Deformation along Tangent Direction}\label{deform}
In this section, we consider a smooth family of immersed submanifolds. We show that if they vary along the tangent directions at each point, then an open region of a nearby submanifold will stay in the time $0$ submanifold. More precisely, we have

\begin{lemma}\label{tandef}
Consider a smooth map $F: U\times I^l\rightarrow \mathbb R^k$, 
where $U$ is an open subset of $\mathbb R^m$ with $m\le k$ and $I=(-1, 1)$. Suppose, for any $(x_0, t_0)\in U\times I^l$, we have $\frac{\partial F(x_0, t_0)}{\partial t}\in (F|_{U\times \{t_0\}})_*(T_{(x_0, t_0)}U)$ and the Jacobian $(\frac{\partial F}{\partial x_1}, \cdots, \frac{\partial F}{\partial x_m})$ has rank $m$.
Then for any proper open subset $U'\subset U$, there exists $0<\ep<1$ and a smooth map $\phi_t: U'\rightarrow U$ such that $F(x, t)=F(\phi_t(x), 0)$ for any $t=(t^1, \cdots, t^l)$ with $|t|<\ep$ and $x\in U'$.
\end{lemma}
\begin{proof}
We first argue it for $l=1$.

We can shrink $U$ a bit, if necessary, to make sure the norms of all derivatives of $F$ are bounded by a constant $C$. 

We assume the coordinates for $U$ are $x_1, \cdots, x_m$. Since $$\frac{\partial F(x_0, t_0)}{\partial t}\in (F|_{U\times \{t_0\}})_*(T_{(x_0, t_0)}U)$$ for any $(x_0, t_0)\in U\times I$ and the Jacobian is of rank $m$, we have a unique set of functions $a_i(x, t)$, depending smoothly on $x$ and $t$, such that $$\dfrac{\partial F}{\partial t}(x, t)=\sum _{i=1}^ma_i(x, t)\dfrac{\partial F(x, t)}{\partial x_i}.$$
Integrate $a_i(x, t)$ along $t$, we define $\phi_t^i(x)=x_i+\int_0^ta_i(x, t)dt$. Hence, the function $\phi_t=(\phi_t^1, \cdots, \phi_t^m)$ solves the ODE $$\dfrac{dF(\phi_t(x), 0)}{dt}=\dfrac{\partial F}{\partial t}(x, t),\,\, \,  \phi_0(x)=x$$ if $\phi_t(x)\in U$. 

Since each $a_i(x, t)$ is a smooth function, it has a uniform bound by shrinking $U$ if necessary and taking $t\in [-\frac{1}{2}, \frac{1}{2}]$. Hence for any $U'\subset U$, we can choose $\ep$ such that for $t<\ep$, we have $$\phi_t(x)=(x_1+\int_0^ta_1(x, t)dt, \cdots, x_m+\int_0^ta_m(x, t)dt)\in U, \, \,\,  \forall x\in U'.$$

When $l>1$, we first obtain smooth functions $a_i^j(x, t)$ such that $$\dfrac{\partial F}{\partial t^j}(x, t)=\sum _{i=1}^ma_i^j(x, t)\dfrac{\partial F(x, t)}{\partial x_i}.$$ Then the function $$\phi_t(x)=(x_1+\sum_{j=1}^l\int_0^{t^j}a_1^j(x, t)dt, \cdots, x_m+\sum_{j=1}^l\int_0^{t^j}a_m^j(x, t)dt)$$solves the equations $$\dfrac{\partial F(\phi_t(x), 0)}{\partial t^j}=\dfrac{\partial F}{\partial t^j}(x, t),\,\, \,  \phi_0(x)=x,\, \, \, i=1, \cdots, l$$ if $\phi_t(x)\in U$. 
\end{proof}

This lemma is particularly useful when the image of each leaf $F(U\times \{t\})$ has unique continuation property. Below are a couple of most common examples which follow from Lemma \ref{tandef}.

\begin{example}\label{geod}
Let $F: I \times I^k\rightarrow M$ satisfying the conditions of Lemma \ref{tandef}, and each $F|_{I\times \{t\}}$ a smooth geodesic. Then the image $F(I\times I^k)$ is a smooth geodesic. 
\end{example}

\begin{example}\label{minsurf}
Let $F: X\times I^k\rightarrow M$  satisfying the conditions of Lemma \ref{tandef}, and each $F|_{X\times \{t\}}$ a parametrized immersed minimal submanifold of $M$. Then the  image $F(X\times I^k)$ is an immersed minimal submanifold. 
\end{example}

The last example is a special case of the above one, and will be used in the next section.
\begin{example}\label{Jcurve}
Let $F: X\times I^k\rightarrow M$  satisfying the conditions of Lemma \ref{tandef}, and each $F|_{X\times \{t\}}$ a pseudoholomorphic map from $(X, J_X)$ to $(M, J_M)$. Then the image $F(X\times I^k)$ is a $J_M$-holomorphic subvariety of $M$, containing $F(X\times \{0\})$ as an open subset. In other words, $F(X\times I^k)$  is an analytic continuation of $F(X\times \{0\})$.
\end{example}

\section{Images of Pseudoholomorphic Maps}
In this section, we use the results in Section \ref{deform}, in particular Example \ref{Jcurve}, to study the images of pseudo\-holomorphic maps. We first recall a couple of results from \cite{ZintJ}. The first is Proposition 2.3 of \cite{ZintJ}, which is an extension of uniqueness of continuity for $J$-holomorphic curves.
\begin{prop}\label{uniquecont}
Let $Y$ be a compact $J$-holomorphic submanifold of the almost complex manifold $(M, J)$, $u: (X, J_1)\rightarrow (M, J)$ is a $(J_1, J)$-holomorphic map with $X$ connected and compact. If $Y$ contains the image of an open subset of $X$, then $u(X)\subset Y$.
\end{prop}

The next, appeared as Lemma 3.10 of \cite{ZintJ},  is the local existence of $J$-holomorphic foliations, an extension of the result for dimension $4$ in \cite{T}.

\begin{lemma}\label{diskfol}
Let $J_1$ be an almost complex structure on $\mathbb C^n$ which agrees with the standard almost complex structure $J_0$ at the origin. Choose an almost Hermitian metric $g$ compatible with $J_1$. There exists a constant $\rho_0$ with the following property. Let $\rho<\rho_0$ and let $U$ be the ball of radius $\rho$ in $\mathbb C^{n-1}$ and $D\subset \mathbb C$ the disk of radius $\rho$. Then there is a diffeomorphism $F: D\times U\rightarrow \mathbb C^n$, and constants $L, L_m$ depending only on $g$ and $J_1$, such that
\begin{itemize}
 \item For all $w\in U$, $F(D_{w})$ is a $J_1$-holomorphic submanifold containing $(0, w)$. Here $D_w:=D\times \{w\}$.
\item For all $w\in U$, dist$((\xi, w); F(\xi, w))\le L\cdot \rho\cdot |\xi|$. 
\item For all $w\in U$, the derivatives of order $m$ of $F$ are bounded by $L_m\cdot \rho$.
\item For any $\kappa\in \mathbb CP^{n-1}$, we can choose $F(D_{0})$ such that it is tangent at the origin to the line $l_{\kappa}\subset \mathbb C^n$ determined by $\kappa$.
\end{itemize}
\end{lemma}

We can now show that the image of a pseudoholomorphic map is $J$-holomorphic using a strategy from \cite{ZintJ}. To show our results, we actually apply an equivalent definition of irreducible $1$-subvarieties \cite{T}. A closed set $C\subset M$ with finite, nonzero $2$-dimensional Hausdorff measure is  an irreducible $1$-subvariety if it has no isolated points, and if the complement of a finite set of points in $C$ is a connected smooth submanifold with $J$-invariant tangent space. The equivalence of these definitions follows from Proposition 6.1 in \cite{T} ({\it c.f.} Theorem \ref{pcaholoopen} in next section). 

\begin{theorem}\label{to4}
Let $(M, J)$ be a connected almost complex $4$-manifold, $(X, J_X)$ a connected closed almost complex manifold, and $f: X\rightarrow M$ a proper pseudo\-holomorphic map. If  $f$ is not a constant map and the image $f(X)\ne M$, then $f(X)$ is an irreducible $J$-holomorphic $1$-subvariety.    
\end{theorem}
\begin{proof}
By Proposition \ref{uniquecont}, we know $f(X)$ cannot contain an open subset of $M$, otherwise $f(X)=M$ which contradicts to our assumption. 

For any $x\in X$, by Lemma \ref{diskfol},  we have an open neighborhood $\mathcal N_x$ of $x$ and a diffeomorphism $F: D\times U\rightarrow \mathcal N_x$ such that each $F(D_w)$ is $J_X$-holomorphic. We further assume $x=F(0, 0)$. We can choose our $F$ such that for any complex direction $\kappa\in T_xX$, we have $F(D_0)$ is tangent to $\kappa$. 

As remarked in the proof of Theorem 3.8 in \cite{ZintJ}, the construction for Lemma \ref{diskfol} would also provide a smooth map $F_0: D\times U_0\rightarrow \mathbb C^n$, such that each $F_0(D_{\kappa})$, $\kappa\in U_0\subset \mathbb CP^{n-1}$, is an embedding whose image is a $J_1$-holomorphic disk which is tangent at the origin to the line $l_{\kappa}\subset \mathbb C^n$ determined by $\kappa$. Moreover, $F_0$ maps the zero section $\{0\}\times U_0$ to $0\in \mathbb C^n$ and $F_0|_{(D\setminus \{0\})\times U_0}$ is a diffeomorphism onto its image. Thus, for some $\kappa\in U$, $F_0(D_{\kappa})$ is not a point, otherwise the open set $F_0(D\times U)$ would be mapped by $f$ to a point. Then, the whole $X$ will be mapped by $f$ to a point by Proposition \ref{uniquecont}.
This contradiction implies that we can find a complex direction $\kappa$, such that for the corresponding $F$ provided by Lemma \ref{diskfol} in last paragraph,  $f( F(D_w))$ is not a point for each $w\in U$.

Hence $f(F(D_0))$ is an irreducible $J$-holomorphic $1$-subvariety. We want to argue that for $w$ small, $f\circ F$ maps an open subset of $D_w$ onto $f(F(D_0))$. These disks might be singular. However, by Theorem 6.2 of \cite{MW}, for any such pair of disks, $f(F(D_0))$ and $f(F(D_w))$, and a point of their intersection, we can find a neighborhood of this point and a $C^1$-diffeomorphism such that the pair of disks are mapped to holomorphic curves. Hence, the intersection of the disks is either an open subset of each other or a set of isolated points. If there is a small $w$ such that $f\circ F$ does not map an open subset of $D_w$ onto $f(F(D_0))$, we know $f(F(D_0))$ and $f(F(D_w))$ intersect positively at finitely many points. Moreover, there are finitely many points of $D_w$ which are mapped to singular points of $f(F(D_0))$. Hence, there is an open subset $D'\subset D$ and a constant $\ep$, such that for $|w|< \ep$, each $f(F(D'_w))$ is an embedded $J$-holomorphic disk. This just implies that for any point $v\in D'$, $\frac{\partial(f\circ F(v, w))}{\partial v}\ne 0$. 

We can further assume 
\begin{equation}\label{nosubm}
\frac{\partial(f\circ F(v, w))}{\partial w_i}\wedge \frac{\partial(f\circ F(v, w))}{\partial v}= 0, \, \, \,  \forall (v, w)\in D'\times \{|w|<\ep\}.
\end{equation}
 Otherwise, at some point $(v, w)\in D'\times \{|w|<\ep\}$, we know $f\circ F$ is a submersion. In particular, $f(X)$ would contain an open subset of $M$. This contradicts to the first sentence in our proof.
 
 By virtue of \eqref{nosubm}, we can apply Lemma \ref{tandef}, in particular its special case Example \ref{Jcurve}. This implies $f(F(D'_w))$ shares an open subset with $f(F(D'_0))$ for $|w|<\ep$. In particular $f\circ F$ maps $D'\times \{|w|<\ep\}$ to a $J$-holomorphic $1$-subvariety.  This in turn implies that the images $f(F(D_0))$ and $f(F(D_w))$ share an open subset for $|w|<\ep$.

Hence $f(\mathcal N'_x)$ has the structure of a $J$-holomorphic $1$-subvariety for an open neighborhood $\mathcal N_x'\subset \mathcal N_x$ of $x$. Since $X$ is compact, we can cover it by finitely many such open subsets. Since $f$ is proper, $f(X)$ is a closed $J$-holomorphic $1$-subvariety.

To prove $f(X)$ is an irreducible $1$-subvariety of $(M, J)$, we choose an irreducible component $Z$ of the $J$-holomorphic $1$-subvariety $f(X)$. Then the preimage $f^{-1}(Z)$ is an open and closed subset of $X$. Since $X$ is connected, we know $f^{-1}(Z)=X$. Therefore, $Z$ is the only irreducible component of $f(X)$, and $f(X)$ is an irreducible $J$-holomorphic $1$-subvariety.
\end{proof}

The argument is extendable to give a systematic study of the notion of higher dimensional $J$-holomorphic subvarieties. In particular, is it true that the image of a proper pseudoholomorphic map between two almost complex mani\-folds always a $J$-holomorphic subvariety? Is the singular set of a pseudo\-holomorphic map a $J$-holomorphic subvariety? We will partially answer these two questions in the paper. 

We have a similar statement as Theorem \ref{to4} when we map from an almost complex $4$-manifold.

\begin{theorem}\label{from4}
Let $(M, J)$ be a connected almost complex manifold, $(X, J_X)$ a connected closed almost complex $4$-manifold, and $f: X\rightarrow M$ a proper pseudoholomorphic map. Then $f(X)$ is an irreducible $J$-holomorphic subvariety.    
\end{theorem}
\begin{proof}
Without loss, we can assume $f$ is not a constant map. If $f$ is somewhere immersed, then the image is an irreducible $J$-holomorphic $2$-subvariety by definition. 

Now we assume $f$ is nowhere immersed. We first cover $X$ by finitely many open subsets $\mathcal N_1, \cdots, \mathcal N_k$ such that for each $\mathcal N_i$, as in Theorem \ref{to4}, there is a diffeomorphism $F_i: D\times U\rightarrow \mathcal N_i$ such that  each $F_i(D_w)$ is $J_X$-holomorphic. We can further assume  our $F_i$ has the property that $f( F_i(D_w))$ is not a point for each $w\in U$.
   
Since $f$ is nowhere immersed, we are in the setting of Lemma \ref{tandef}, more precisely Example \ref{Jcurve}. It implies that $f(\mathcal N_i)$ is a $J$-holomorphic $1$-subvariety.  Hence $f(X)$ is a closed $J$-holomorphic $1$-subvariety of $(M, J)$.

It is irreducible since the preimage $f^{-1}(Z)$ of any irreducible component $Z$ of $f(X)$ is an open and closed subset of $X$. By the connectedness of $X$, we know $Z$ is the only irreducible component of $f(X)$. 
\end{proof}

In the case when the image of the pseudoholomorphic map $f: X\rightarrow M$ in Theorem \ref{from4} is a $1$-subvariety, we have very strong restrictions on $X$. 
\begin{prop}\label{preim}
Let $(M, J)$ be an almost complex manifold, $(X, J_X)$ a closed almost complex $4$-manifold, and $f: X\rightarrow M$ a proper pseudoholomorphic map. If $f(X)$ is a $J$-holomorphic $1$-subvariety, then for any point  $x\in X$, there is a $J_X$-holomorphic $1$-subvariety passing through. 
\end{prop}
To show this, we introduce positive cohomology assignment. This is a notion introduced by Taubes, which plays the role of intersection number of a set with suitable local open disks. We assume $(X, J)$ is an almost complex manifold, and $C\subset X$ is merely a subset at this moment. Let $D\subset \mathbb C$ be the standard unit disk. A map $\sigma: D\rightarrow X$ is called {\it admissible} if $C$ intersects the closure of $\sigma(D)$ inside $\sigma(D)$. The following is extracted from Section 6.1(a) of \cite{T}.

\begin{definition}\label{PCA}
A positive cohomology assignment to the set $C$ is an assignment of an integer, $I(\sigma)$, to each admissible map $\sigma: D\rightarrow X$. Furthermore, the following criteria have to be met: 
\begin{enumerate}
\item If $\sigma: D\rightarrow X\setminus C$, then $I(\sigma)=0$. 

\item If $\sigma_0, \sigma_1: D\rightarrow X$ are admissible and homotopic via an admissible homotopy (a homotopy $h:[0, 1]\times D\rightarrow X$ where $C$ intersects the closure of Image$(h)$ inside Image$(h)$), then $I(\sigma_0)=I(\sigma_1)$.

\item Let $\sigma: D\rightarrow X$ be admissible and let $\theta: D\rightarrow D$ be a proper, degree $k$ map. Then $I(\sigma\circ \theta)=k\cdot I(\sigma)$.

\item Suppose that $\sigma: D\rightarrow X$ is admissible and that $\sigma^{-1}(C)$ is contained in a disjoint union $\cup_iD_i\subset D$ where each $D_i=\theta_i(D)$ with $\theta_i: D_i\rightarrow D$ being an orientation preserving embedding. Then $I(\sigma)=\sum_iI(\sigma\circ \theta_i)$.

\item If $\sigma: D\rightarrow X$ is admissible and a $J$-holomorphic embedding with $\sigma^{-1}(C)\ne \emptyset$, then $I(\sigma)>0$.
\end{enumerate}
\end{definition}

We use this notion to show certain sets are pseudoholomophic subvarieties. The following is Proposition 6.1 of \cite{T}.

\begin{theorem}\label{pcaholo}
Let $(X, J)$ be a $4$-dimensional almost complex manifold and let $C\subset X$ be a closed set with finite $2$-dimensional Hausdorff measure and a positive cohomology assignment. Then $C$ supports a compact $J$-holomorphic $1$-subvariety. 
\end{theorem}

\begin{proof}[Proof of Proposition \ref{preim}]
When $f(x)$ is a smooth point of the image $1$-subvariety, it follows from Theorem 1.2 (or Proposition 5.1) in \cite{ZintJ} that $f^{-1}(f(x))$ is a $1$-subvariety passing through $x$. In fact, this argument could be extended to singular points. 

Let $y=f(x)\in M$. To show $f^{-1}(y)$ is a $J_X$-holomorphic $1$-subvariety, we apply Theorem \ref{pcaholo}. It is a closed set with finite $2$-dimensional Hausdorff measure following from Lemma 2.1 and Proposition 2.4. To show it has a positive cohomology assignment, we assign an integer $IC(\sigma)$ to any admissible map $\sigma: D\rightarrow X$. Its composition with $f: X\rightarrow M$, $f\circ \sigma: D\rightarrow M$, is also admissible with respect to $y\in f(X)$. Now the preimage $(f\circ \sigma)^{-1}(y')$ of a generic point $y'\in f\circ \sigma(D)$ is a finite set of signed points. We define $IC(\sigma)$ to be the sum of these signs. The argument of Proposition 3.3 of \cite{ZintJ} then implies $IC(\sigma)$ defines a positive cohomology assignment to $f^{-1}(y)$. 
\end{proof}

\section{Singularity Subsets of Pseudoholomorphic Maps}
We could then study $J$-holomorphic $2$-subvarieties in detail. 
The case of pseudoholomorphic maps between two almost complex $4$-manifolds corresponds to birational geometry for almost complex manifolds. It is discussed in \cite{ZintJ}, see Theorem 1.4. First comes a generalization of this result when the target manifold is assumed of dimension $6$ or higher. Although it is actually implied by Theorem \ref{sinJhol} later, we write it separately as it does not apply geometric measure theory argument.

\begin{prop}\label{wd2sv}
Let $(M, J)$ be a connected almost complex manifold, $(X, J_X)$ a connected closed almost complex $4$-manifold, and $f: X\rightarrow M$ a pseudoholomorphic map. If $f$ is somewhere immersed, then $f$ is immersed at an open dense subset of $X$.
\end{prop}
\begin{proof}
Apparently, the set of points where $f$ is immersed is an open subset of $X$. To show it is dense, we prove it by contradiction. That is, we assume there is an open subset $\mathcal N\subset X$ such that $f$ is not immersed at any point of it. We can further assume that $\mathcal N$ is the largest among such open subsets. Notice that $f$ is not immersed at all points of $\overline{\mathcal N}$.

 First, $f(\mathcal N)$ is not a point, otherwise $f$ is a constant map by Lemma \ref{uniquecont}. This will contradict to the assumption that $f$ is somewhere immersed. Now we choose a point $x\in \overline{\mathcal N}\setminus \mathcal N$. As in Theorem \ref{to4}, we can choose an open neighborhood $\mathcal N_x\subset \mathcal N$ of $x$ and a diffeomorphism $F: D\times U\rightarrow \mathcal N_x$ such that $x=F(0, 0)$ and each $F(D_w)$ is $J_X$-holomorphic. We can further assume  our $F$ has the property that $f( F(D_w))$ is not a point and $F(D_w)\cap\mathcal N$ is an open subset of $F(D_w)$ for each $w\in U$.

As $f$ is not immersed at any point of $F(D_w)\cap\mathcal N_x$, by applying of Lemma \ref{tandef}, more precisely Example \ref{Jcurve}, we know each $f((F(D_w)\cup F(D_0))\cap\mathcal N)$ is an analytic extension of $f(F(D_0)\cap\mathcal N)$. This implies that each $f((F(D_w)\cup F(D_0))$ is an analytic extension of  $f(F(D_0))$. This in turn implies that $f(\mathcal N_x)$ is a $J$-holomorphic $1$-subvariety and $f$ is not immersed at any point of $\mathcal N_x$,  contradicting to our assumed maximality of $\mathcal N$.  This contradiction implies that $f$ is immersed at an open dense subset of $X$. 
\end{proof}
 In particular, it implies that other than a subset of $4$-dimensional Hausdorff measure zero, a $2$-subvariety is a smooth submanifold of dimension $4$ with $J$-invariant tangent space.

In fact, we can show that the singularity subset is  a union of $J_X$-holomorphic $1$- and $0$-subvarieties, following the argument of Proposition 2.1 in \cite{CZ2} and Theorem 5.5 in \cite{ZintJ}.

To get this kind of description, we apply the geometric measure theory as in \cite{T}. We quote Theorem 5.2 in \cite{iccm} (and its apparent extension to higher dimensions) below, which is a generalization of Theorems 1.2 and 3.8 in \cite{ZintJ}.

\begin{theorem}\label{ICdim4open}
Suppose $(M, J)$ is an almost complex $2n$-manifold, 
and  $Z_2$ is a codimension $2$ embedded almost complex submanifold. 
\begin{enumerate}
\item Let $(M_1, J_1)$ be a  connected almost complex $4$-manifold and $u: M_1\rightarrow M$ a pseudoholomorphic map such that $u(M_1)\nsubseteq Z_2$. Then $u^{-1}(Z_2)$ supports a  (possibly open) $J_1$-holomorphic $1$-subvariety in $M_1$.

\item Let $(M_1, J_1)$ be a  connected almost complex manifold of dimension $2k<2n$ and $u: M_1\rightarrow M$ a pseudoholomorphic map such that $u(M_1)\nsubseteq Z_2$. Then $u^{-1}(Z_2)$ is a closed set with finite $(2k-2)$-dimensional Hausdorff measure and a positive cohomology assignment. 
\end{enumerate}
\end{theorem}

A couple of remarks are in order.  First, definition of positive cohomology assignment is in Definition \ref{PCA}.

Second, here we use a generalized notion of $1$-subvarieties to include more interesting objects when $M_1$ is non-compact. 
A (possibly open) $1$-subvariety in an almost complex manifold $M$ is a closed subset $C$ such that 
\begin{itemize}
\item There is a smooth complex curve $C_0$ (not necessarily compact or connected) with a proper pseudoholomophic map $\phi: C_0\rightarrow M$ with $C=\phi(C_0)$. 
\item There is a countable set $\Lambda_0\subset C_0$ which has no accumulation points and is such that $\phi$ embeds $C_0-\Lambda_0$. 
\item The restriction to any open subset with compact closure has finite $2$-dimensional Hausdorff measure.
\end{itemize}

When $M$  has a symplectic form $\omega$ and $J$ is tamed by $\omega$, a $1$-subvariety has  finite energy which means the integral of $\phi^*\omega$ over $C_0$ is finite. We can also associate multiplicities to each component as in our original definition for compact almost complex manifolds, but we leave it as the above definition since we do not discuss homology class of subvarieties in this paper. 

The upshot of the above two remarks is the following result which is Proposition 7.1 in \cite{Tsd} and is a generalization of Proposition 6.1 in \cite{T}.

\begin{theorem}\label{pcaholoopen}
Let $(X, J)$ be a $4$-dimensional almost complex manifold. 
Suppose that $C\subset X$ is a closed set with the following properties:
\begin{itemize}
\item The restriction of $C$ to any open $X'\subset X$ with compact closure has finite $2$-dimensional Hausdorff measure. 
\item $C$ has a positive cohomology assignment. 
\end{itemize}

Then $C$ supports a $J$-holomorphic $1$-subvariety. 
\end{theorem}

Now we are ready to state our description of singularity subset of pseudoholomorphic maps. 

\begin{theorem}\label{sinJhol}
Let $(M, J)$ be a connected almost complex manifold, $(X, J_X)$ a connected closed almost complex $4$-manifold, and $f: X\rightarrow M$ a pseudoholomorphic map. If $f$ is somewhere immersed, then $f$ is immersed in an open dense subset of $X$ and its singularity subset is a union of finitely many $J_X$-holomorphic $1$- and $0$-subvarieties. In particular, it implies that other than a union of finitely many $1$- and $0$-subvarieties subset, a $2$-subvariety is a smooth almost complex $4$-manifold. 
\end{theorem}
\begin{proof}
The almost complex structure $J$ induces a complex vector bundle structure on $TM$. Choose covers of $M$ such that $TM$ is trivial on each open subset of the cover.  Hence, we can assume our target is an open subset $U\subset M$, such that $(TU, J)$ is a trivial complex vector bundle with fiber coordinates  $(z_1, \cdots, z_n)\in \mathbb C^n$.  Let $\{\sigma_1, \sigma_2\}$ be a choice of a pair of elements in $\{1, \cdots, n\}$. For each such choice of $\{\sigma_1, \sigma_2\}$, we look at the projection from $\mathbb C^n$ to coordinates $(z_{\sigma_1}, z_{\sigma_2})$. Let $X_U$ be the set $f^{-1}(U)\subset X$. Denote the projection of $df|_{X_U}$ to $\C^2_{z_{\sigma_1}, z_{\sigma_2}}$ components of fibers by $pr_{\sigma_1\sigma_2}df$.

We thus have a complex vector bundle over $X_U\times U$ whose fiber over $(x, f(x))$ is the complex vector space of all complex linear maps $L_{\sigma_1\sigma_2}: T_xX\rightarrow \C^2$.  By taking fiberwise complex determinant, we have a complex line bundle $\mathcal L_{\sigma_1\sigma_2}$ over $X_U\times U$, whose fibers are $\det L_{\sigma_1\sigma_2}: \Lambda_{\C}^2T_xX\rightarrow \Lambda_{\C}^2\C^2\cong \C$. The total space of $\mathcal L_{\sigma_1\sigma_2}$ has a standard almost complex structure, see the proof of Theorem 5.5 in \cite{ZintJ}.

 Then the pseudoholomorphic map $f$ induces pseudoholomorphic maps $f_{\mathcal L_{\sigma_1\sigma_2}}(x)=(x, f(x), \det ((pr_{\sigma_1\sigma_2}df)_x)_{\mathbb C})$ from $X_U$ to $\mathcal L_{\sigma_1\sigma_2}$.  Hence, the singularity subset of $f$ (in $X_U$) is the intersection of $f^{-1}_{\mathcal L_{\sigma_1\sigma_2}}(X_U\times U \times \{0\})$ for all possible subsets $\{\sigma_1, \sigma_2\}$. Applying Theorem 3.8 of \cite{ZintJ}, we know each $f^{-1}_{\mathcal L_{\sigma_1\sigma_2}}(X_U\times U \times \{0\})$ is a closed subset with finite $2$-dimensional Hausdorff measure and positive cohomology assignment if it is not the whole $X_U$.  Since $f$ is somewhere immersed and $X$ is connected, there must be $\{\sigma_1, \sigma_2\}$ such that $f^{-1}_{\mathcal L_{\sigma_1\sigma_2}}(X_U\times U \times \{0\})$ is not the whole $X_U$ (but possibly it is empty). By Theorem \ref{pcaholoopen}, $f^{-1}_{\mathcal L_{\sigma_1\sigma_2}}(X_U\times U \times \{0\})$ is a $J_X$-holomorphic $1$-subvariety. Since $\mathcal S_f$ is the intersection of $f^{-1}_{\mathcal L_{\sigma_1\sigma_2}}(X_U\times U \times \{0\})$ for all possible pairs $\{\sigma_1, \sigma_2)$, we know $\mathcal S_f$ is a union of $1$- and $0$-subvarieties by positivity of intersection in dimension $4$ \cite{MW}.
\end{proof}
It is clear that the above argument also works when $\dim X=2m>4$ and implies that the singularity subset $\mathcal S_{f}$ has finite $2(m-1)$-dimensional Hausdorff measure. 
\begin{prop}\label{sinJhol>4}
Let $(M, J)$ be a connected almost complex manifold, $(X, J_X)$ a connected closed almost complex $2m$-manifold, and $f: X\rightarrow M$ a pseudoholomorphic map. If $f$ is somewhere immersed, then $f$ is immersed in an open dense subset of $X$ and its singularity subset is the intersection of finitely many closed subsets with finite $2(m-1)$-dimensional Hausdorff measure and positive cohomology assignment. 
\end{prop}

When the map $f: X\rightarrow M$ is nowhere immersed, every point of $X$ is a singular point of the map. However, we could similar study the subset of $X$ such that $df_p$ is not of maximal rank. Denote this subset of $X$ to be $\mathcal S'_f$. We assume $\dim X=4$. If $f(X)$ is not a point, by Theorem \ref{from4}, we know the image $f(X)$ is a pseudoholomorphic $1$-subvariety. There are finitely many singularities in $f(X)$ and preimage of these points are in $\mathcal S'_f$. By Proposition \ref{preim}, these are pseudoholomorphic $1$-subvariety. A similar argument as in Theorem \ref{sinJhol}, but now trivialize tangent bundle $(TU,J_X)$ with $U\subset X$, would imply that $\mathcal S'_f$ is a union of finitely many $J_X$-holomorphic $1$- and $0$-subvarieties. 

We now look at the map $f$ restricting at preimage of the smooth part of $f(X)$, and discuss its critical values and critical points. The preimage of a regular value is a smooth $1$-subvariety in $X$ and none of its points is in $\mathcal S'_f$. For the preimage of a critical value $p$ in the smooth points of $f(X)$, it is also a $1$-subvariety. Exactly the singular points of the  $1$-subvariety $f^{-1}(p)$ are in $\mathcal S'_f$. The critical values must be isolated as otherwise it is the image of a branched covering from a $1$-subvariety, contradicting to Sard's theorem that the critical values is of measure zero. Hence, the critical points are also isolated for this restricted map. 

To summarize, we have

\begin{theorem}\label{sinJh}
Let $(M, J)$ be a connected almost complex manifold, $(X, J_X)$ a connected closed almost complex $4$-manifold, and $f: X\rightarrow M$ a pseudoholomorphic map. If $f$ is not a constant map, then the subset $\mathcal S'_f$ of $X$ such that $df_p$ is not of maximal rank  is a union of finitely many $J_X$-holomorphic $1$- and $0$-subvarieties. Moreover, when $f$ is nowhere immersed, the image of $\mathcal S'_f$ forms a $0$-subvariety in $M$. 
\end{theorem}

\end{document}